\documentclass{amsart}
\usepackage[dvips,final]{graphics}
\usepackage{array}
\usepackage{arydshln}
\usepackage[makeroom]{cancel}
 \usepackage[all]{xy}
 \usepackage{url}
\usepackage{multirow, blkarray}
\usepackage{booktabs}
\usepackage{textcomp}
 \usepackage[final]{epsfig}
 \usepackage{color}
\usepackage[T1]{fontenc}      
\usepackage[english,french]{babel}
\usepackage[utf8]{inputenc}
\usepackage{blindtext}

\usepackage{amsfonts,amscd,array, mathdots, epigraph}
\usepackage{amsmath}
\usepackage{amssymb}
\usepackage{amsthm}
\usepackage{mathrsfs}
\usepackage{stmaryrd}
\usepackage{slashbox}
\usepackage{diagbox}
\usepackage{enumitem}

\usepackage{ulem}
\usepackage{tikz}
\usepackage{xcolor}
\usepackage{multicol}
\usepackage{tcolorbox}
\definecolor{ufogreen}{rgb}{0.24, 0.82, 0.44}

\usepackage{hyperref}

\begin{document}


\newtheorem{theorem}{Théorème}[section]
\newtheorem{theore}{Théorème}
\newtheorem{definition}[theorem]{Définition}
\newtheorem{proposition}[theorem]{Proposition}
\newtheorem{corollary}[theorem]{Corollaire}
\newtheorem{con}{Conjecture}
\newtheorem*{remark}{Remarque}
\newtheorem*{remarks}{Remarques}
\newtheorem*{pro}{Problème}
\newtheorem*{examples}{Exemples}
\newtheorem*{example}{Exemple}
\newtheorem{lemma}[theorem]{Lemme}
\newtheorem{theorembis}[theorem]{Theorem}
\newtheorem{lemmabis}[theorem]{Lemma}

\newtcolorbox{attentionbox}{
  colback=yellow!10,
  colframe=red!70!black,
  title=Warning,
  fonttitle=\bfseries
}


\title{Une curieuse égalité entre deux sommes de produits de coefficients binomiaux}
\label{00}

\author{Flavien Mabilat}

\date{}


\address{
}
\def\emailaddrname{{\itshape Courriel}}
\email{flavien.mabilat@univ-reims.fr}

\maketitle

\selectlanguage{french}

\begin{abstract}
On va montrer dans ce texte que, pour tous entiers naturels $n$ et $l$, l'égalité suivante est vérifiée : 
\[\sum_{i=0}^{l} {n-i \choose i}{l+i \choose 2i+1}=\sum_{i=0}^{l} {n-i \choose i-1}{l+i \choose 2i}.\]
\noindent On traitera d'abord le cas où $l \leq n$ pour lequel les deux sommes ne contiennent que des coefficients binomiaux classiques. Ensuite, on se placera dans le cadre général en utilisant les coefficients binomiaux généralisés.
\end{abstract}

\selectlanguage{english}
\begin{abstract}

We will show in this text that, for all non-negative integers $n$ and $l$, the following equality is verified:
\[\sum_{i=0}^{l} {n-i \choose i}{l+i \choose 2i+1}=\sum_{i=0}^{l} {n-i \choose i-1}{l+i \choose 2i}.\]
\noindent We will first address the case where $l \leq n$, for which both sums contain only classical binomial coefficients. Then, we will consider the general framework using generalized binomial coefficients.

\end{abstract}

\selectlanguage{french}

\thispagestyle{empty}

\noindent {\bf Mots clés:} coefficients binomiaux, continuant, polynômes de Tchebychev de seconde espèce.   

\section{Introduction}
\label{intro}

Les coefficients binomiaux apparaissent dans de nombreuses branches des mathématiques. Outils incontournables de la combinatoire, on les retrouve dans une myriade de résultats de dénombrement bien connus, comme par exemple le théorème de Sperner ou le théorème de Cayley (voir par exemple \cite{AZ}). Cela dit, ils se font également les compagnons de route d'objets en apparence plus éloignés des questions de comptage. Les coefficients binomiaux surgissent notamment dans l'expression des polynômes de Tchebychev, dans la démonstration de Bernstein du théorème d'approximation de Weierstrass ou bien encore dans la formule de Leibniz sur les dérivées.
\\
\\ \indent À la lueur de ce large spectre d'application, il n'est pas étonnant que de nombreuses propriétés des coefficients binomiaux aient été recherchées. Parmi celles-ci, on peut distinguer deux grandes familles : les propriétés arithmétiques et les formules de sommation. Dans la première catégorie, on peut notamment citer le théorème de de Lucas qui exprime le reste de la division du coefficient binomial ${n \choose k}$ par un nombre premier $p$ en fonction du développement en base $p$ des entiers $n$ et $k$ (voir \cite{L} section XXI). Pour ce qui est des formules sur les sommes, on dispose d'un large éventail de choix allant du triangle de Pascal au binôme de Newton en passant par un grand nombre de formules plus ou moins célèbres, telle la formule de Vandermonde (voir \cite{K} section 1.2.6 I).
\\
\\ \indent On se propose ici de s'inscrire dans cette deuxième voie, en démontrant les deux formules ci-dessous, à l'aide d'une preuve qui se veut la moins calculatoire possible. Dans tout ce qui suit, on adoptera la convention usuelle suivante : pour tout $n \geq 0$, $k \geq 1$ et $0 \leq l < k$ on pose ${n \choose -k}={l \choose k}=0$. Notons qu'avec cette convention, le triangle de Pascal fonctionne pour les coefficients binomiaux de la forme ${m \choose 0}$ et ${m \choose m}$, avec $m \geq 1$. Par ailleurs, si $x$ est un réel, on note $E[x]$ la partie entière de $x$.

\begin{theorem}
\label{11}

Pour tous entiers naturels $n$ et $0 \leq l \leq n$, on dispose de l'égalité :
\[\sum_{i=0}^{l} {n-i \choose i}{l+i \choose 2i+1}=\sum_{i=0}^{l} {n-i \choose i-1}{l+i \choose 2i}.\]

\end{theorem}

En réalité, la condition $l \leq n$ est inutile. Cependant, pour la lever, il faut donner un sens à ${n-i \choose i}$ lorsque $n-i<0$. On pose donc, pour $n<0$ et $k \in \mathbb{Z}$, 
\[{n \choose k}:=\left\{
    \begin{array}{ll}
        (-1)^{k}{k-n-1 \choose k} & \mbox{si } k \geq 0; \\
        0 & \mbox{sinon}.
    \end{array}
\right.  \\ \]
\noindent Cette définition peut sembler étonnante de prime abord mais elle est justifiée par les égalités ci-dessous :
\[(-1)^{k}{k-n-1 \choose k}=(-1)^{k}\frac{(k-n-1)!}{k!(-n-1)!}=(-1)^{k}\frac{(k-n-1)\ldots(-n)}{k!}=\frac{n(n-1)\ldots(n-k+1)}{k!}.\]
\noindent Par ailleurs, avec cette convention, le triangle de Pascal fonctionne pour le coefficient binomial ${0 \choose 0}$. Le triangle de Pascal est donc valide pour tous les coefficients binomiaux ${n \choose k}$ avec $0 \leq k \leq n$.
\\
\\ \indent Muni de cette définition, on peut énoncer le résultat ci-dessous, qui fournit une jolie formule qui, à notre connaissance, est inédite :

\begin{theorem}
\label{12}

Pour tous entiers naturels $n$ et $l$, on dispose de l'égalité :
\[\sum_{i=0}^{l} {n-i \choose i}{l+i \choose 2i+1}=\sum_{i=0}^{l} {n-i \choose i-1}{l+i \choose 2i}.\]

\end{theorem}

Cette égalité peut sembler assez curieuse de prime abord, du fait, d'une part, de l'absence de conditions sur $n$ et $l$, et, d'autre part, de sa symétrie presque parfaite. En effet, en modifiant seulement, et de façon très légère, les valeurs apparaissant en bas des coefficients binomiaux, il y avait a priori peu de chance d'obtenir une égalité. Notons que cette intuition initiale peut se retrouver en considérant des sous-sommes de chaque côté de l'égalité, c'est-à-dire les sommes de la forme $\sum_{i \in I} {n-i \choose i}{l+i \choose 2i+1}$ et $\sum_{j \in J} {n-j \choose j-1}{l+j \choose 2j}$ avec $I,J \subset [\![1;l]\!]$, puisqu'en général aucune des sous-sommes de gauche n'est égale à une des sous-sommes de droite (à l'exception évidemment des sous-sommes égales 0 ou égales à la valeur commune des sommes initiales). On donne un exemple de ce phénomène dans la section \ref{num}.

\section{Preuve de la formule}

\subsection{Le polynôme continuant}

On va donner dans cette sous-partie quelques éléments qui nous seront utiles dans la suite.
\\
\\\indent On commence par définir un polynôme célèbre, lié à l'origine aux fractions continues (voir par exemple \cite{CO} pour plus de détails). On pose $K_{-1}:=0$ et $K_{0}:=1$. Soient $n \in \mathbb{N}^{*}$ et $(a_{1},\ldots,a_{n}) \in \mathbb{C}^{n}$. On note : \[K_{n}(a_{1},\ldots,a_{n}):=
\left|
\begin{array}{cccccc}
a_1&1&&&\\[4pt]
1&a_{2}&1&&\\[4pt]
&\ddots&\ddots&\!\!\ddots&\\[4pt]
&&1&a_{n-1}&\!\!\!\!\!1\\[4pt]
&&&\!\!\!\!\!1&\!\!\!\!a_{n}
\end{array}
\right|.\] $K_{n}(a_{1},\ldots,a_{n})$ est le continuant (ou continuant négatif) de $a_{1},\ldots,a_{n}$. Ici, on aura uniquement besoin de considérer le cas où tous les $a_{i}$ sont égaux et on notera dans ce cas le continuant $K_{n}(a_{1})$ au lieu de $K_{n}(a_{1},\ldots,a_{1})$. Lorsque l'on se place dans ce cas, on obtient un polynôme d'une variable, à coefficients entiers, de degré $n$. En développant le déterminant suivant la première colonne, on obtient immédiatement, pour tout $n \geq 2$, l'égalité :
\begin{equation}
\label{p}
\tag{$\star$}
K_{n}(X)=X~K_{n-1}(X)-K_{n-2}(X).
\end{equation}
\noindent Notons que cette dernière est toujours valide si $n=1$.
\\
\\ \indent Ces polynômes peuvent être utilisés pour exprimer les coefficients de certaines matrices que l'on va expliciter. Soient $n \in \mathbb{N}$ et $x \in \mathbb{C}$. On pose $M_{n}(x):=\begin{pmatrix}
    x & -1 \\
    1  & 0 
   \end{pmatrix}^{n}$. On dispose de l'égalité classique ci-dessous :
	
\begin{proposition}
\label{21}

Soient $n \in \mathbb{N}^{*}$ et $x \in \mathbb{C}$. On a $M_{n}(x)=\begin{pmatrix}
    K_{n}(x) & -K_{n-1}(x) \\
    K_{n-1}(x)  & -K_{n-2}(x)
   \end{pmatrix}$.

\end{proposition}

\begin{proof}

On raisonne par récurrence sur $n$. Si $n=1$ le résultat est immédiat. Supposons qu'il existe un entier $n \geq 1$ tel que la formule est vraie au rang $n$. On a :
\begin{eqnarray*}
M_{n+1}(x) &=& M_{n}(x)\begin{pmatrix}
    x & -1 \\
    1  & 0 
   \end{pmatrix} \\
	         &=& \begin{pmatrix}
    K_{n}(x) & -K_{n-1}(x) \\
    K_{n-1}(x)  & -K_{n-2}(x)
   \end{pmatrix}\begin{pmatrix}
    x & -1 \\
    1  & 0 
   \end{pmatrix}~~({\rm hypoth\grave{e}se~de~r\acute{e}currence}) \\
	         &=& \begin{pmatrix}
    xK_{n}(x)-K_{n-1}(x) & -K_{n}(x) \\
    xK_{n-1}(x)-K_{n-2}(x)  & -K_{n-1}(x)
   \end{pmatrix} \\
	         &=& \begin{pmatrix}
    K_{n+1}(x) & -K_{n}(x) \\
    K_{n}(x) & -K_{n-1}(x)
   \end{pmatrix}~~({\rm \acute{e}galit\acute{e}~(\star)}).
\end{eqnarray*}

\end{proof}

À l'instar des polynômes $K_{n}(x)$, les matrices $M_{n}(x)$ sont également des cas particuliers. On peut, bien entendu, considérer les matrices plus générales ci-dessous :
\[M_{n}(a_1,\ldots,a_n):=\begin{pmatrix}
    a_{n} & -1 \\
    1  & 0 
   \end{pmatrix}\ldots\begin{pmatrix}
    a_{1} & -1 \\
    1  & 0 
   \end{pmatrix}=\begin{pmatrix}
    K_{n}(a_{1},\ldots,a_{n}) & -K_{n-1}(a_{2},\ldots,a_{n}) \\
    K_{n-1}(a_{1},\ldots,a_{n})  & -K_{n-2}(a_{2},\ldots,a_{n-1})
   \end{pmatrix}.\]
\noindent Ces matrices ont de nombreuses applications. On peut par exemple montrer que, pour tout $A$ dans $SL_{2}(\mathbb{Z})$, il existe des entiers naturels non nuls $n,a_{1},\ldots,a_{n}$ tels que $A=M_{n}(a_{1},\ldots,a_{n})$. Cette écriture n'étant pas unique, on est amené à chercher les différentes écritures d'une matrice donnée ainsi que le nombre de telles écritures de taille fixée (voir \cite{M1,M2,O}).
\\
\\\indent On va maintenant chercher une expression sous forme de somme de $K_{n}(X)$. Pour obtenir celle-ci, on peut utiliser les polynômes de Tchebychev de seconde espèce. Ces derniers sont une famille $(U_{n})$ de polynômes de $\mathbb{Z}[X]$ vérifiant, pour tout $n \in \mathbb{N}$, $U_{n+2}(X)=2X U_{n+1}(X)-U_{n}(X)$, $U_{0}(X)=1$ et $U_{1}(X)=2X$. En utilisant l'égalité ($\star$) on constate, que $K_{n}(2X)$ vérifie ces conditions. Par conséquent, on a, pour tout $n \geq 0$, $K_{n}(X)=U_{n}\left(\frac{X}{2}\right)$. En utilisant l'expression classique des polynômes $U_{n}$, on obtient l'expression de $K_{n}(X)$ contenue dans la proposition ci-dessous. Cela dit, ici, on souhaite avoir une preuve complète du théorème \ref{11}. Aussi, on fournit une preuve détaillée de l'égalité recherchée. 

\begin{proposition}
\label{22}

Soit $n \geq 0$, $K_{n}(X)=\sum_{k=0}^{E[\frac{n}{2}]} (-1)^{k}\binom{n-k}{k}X^{n-2k}$.

\end{proposition}

\begin{proof}

Cela se prouve par récurrence sur $n$. En effet, la  formule est vraie pour $n=0$ et pour $n=1$. Supposons qu'il existe un entier positif $n$ tel que la formule est vraie pour $n$ et $n-1$. On suppose $n$ pair. En utilisant ($\star$), on obtient :
\begin{eqnarray*}
K_{n+1}(X) &=& XK_{n}(X)-K_{n-1}(X) \\
                    &=& \sum_{k=0}^{E[\frac{n}{2}]} (-1)^{k}\binom{n-k}{k}X^{n+1-2k}-\sum_{k=0}^{E[\frac{n-1}{2}]} (-1)^{k}\binom{n-1-k}{k}X^{n-1-2k} \\
										&=&  \sum_{k=0}^{\frac{n}{2}} (-1)^{k}\binom{n-k}{k}X^{n+1-2k}-\sum_{k=0}^{\frac{n}{2}-1} (-1)^{k}\binom{n-1-k}{k}X^{n-1-2k}~({\rm car}~n~{\rm est~pair}) \\
										&=&  \sum_{k=0}^{\frac{n}{2}} (-1)^{k}\binom{n-k}{k}X^{n+1-2k}-\sum_{l=1}^{\frac{n}{2}} (-1)^{l-1}\binom{n-l}{l-1}X^{n+1-2l} \\ 
										&=&  \sum_{k=0}^{\frac{n}{2}} (-1)^{k}\binom{n-k}{k}X^{n+1-2k}+\sum_{l=1}^{\frac{n}{2}} (-1)^{l}\binom{n-l}{l-1}X^{n+1-2l} \\
										&=&  X^{n+1}+\sum_{k=1}^{\frac{n}{2}} (-1)^{k}\left(\binom{n-k}{k}+\binom{n-k}{k-1}\right)X^{n+1-2k} \\
										&=&  X^{n+1}+\sum_{k=1}^{\frac{n}{2}} (-1)^{k}\binom{n+1-k}{k}X^{n+1-2k}~({\rm triangle~de~Pascal}) \\
										&=&  \sum_{k=0}^{\frac{n}{2}} (-1)^{k}\binom{n+1-k}{k}X^{n+1-2k}. \\
\end{eqnarray*}

\noindent On procède de façon analogue si $n$ est impair. Cela prouve la formule par récurrence.

\end{proof}
  
\subsection{Preuve du premier théorème}

On va maintenant prouver l'égalité annoncée dans l'introduction.

\begin{proof}[Démonstration du théorème \ref{11}]

Si $l=0$ alors les deux sommes ne contiennent chacune qu'un seul terme valant 0. Si $n=0$ alors $l=0$ et l'égalité est vérifiée. On suppose maintenant que $n,l \geq 1$. On distingue deux cas.
\\
\\i) On suppose que $2l+1 \leq n$, c'est-à-dire $1 \leq l \leq \frac{n-1}{2}$. On pose $m:={\rm min}\left(E\left[\frac{n}{2}\right],E\left[\frac{2l-1}{2}\right]\right)$ et $r:={\rm min}\left(E\left[\frac{n-1}{2}\right],l\right)$. On a $m=l-1$ et $r=l$. Par la proposition \ref{21}, on a les deux égalités ci-dessous :
\begin{eqnarray*}
M &:=& M_{n}(X)[M_{2l+1}(X)]^{-1} \\
  &=& M_{n-2l-1}(X) \\
	&=& \begin{pmatrix}
    K_{n-2l-1}(X) & -K_{n-2l-2}(X) \\
    K_{n-2l-2}(X)  & -K_{n-2l-3}(X)
   \end{pmatrix}.
\end{eqnarray*}
\begin{eqnarray*}
M &=& M_{n}(X)[M_{2l+1}(X)]^{-1} \\
  &=& \begin{pmatrix}
    K_{n}(X) & -K_{n-1}(X) \\
    K_{n-1}(X)  & -K_{n-2}(X)
   \end{pmatrix}\begin{pmatrix}
    -K_{2l-1}(X) & K_{2l}(X) \\
    -K_{2l}(X)  & K_{2l+1}(X)
   \end{pmatrix} \\
	&=& \begin{pmatrix}
    -K_{n}(X)K_{2l-1}(X)+K_{n-1}(X)K_{2l}(X) & \ldots \\
    \ldots  & \ldots
   \end{pmatrix}.
\end{eqnarray*}

\noindent Par conséquent, on a $K_{n-2l-1}(X)=-K_{n}(X)K_{2l-1}(X)+K_{n-1}(X)K_{2l}(X)$.
\\
\\Comme $n-2l-1 \leq n$, $K_{n-2l-1}(X)$ n'a pas de terme de degré $n+1$. Ainsi, le coefficient de degré $n+1$ de $K_{n}(X)K_{2l-1}(X)$, que l'on note $u$, est égal au coefficient de degré $n+1$ de $K_{n-1}(X)K_{2l}(X)$, que l'on note $v$.
\\
\\L'entier $u$ est la somme des produits du coefficient de degré $n-2i$ de $K_{n}(X)$ et du coefficient de $K_{2l-1}(X)$ de degré $2i+1$, pour tout $0 \leq i \leq m$. Par la proposition \ref{22}, on a :
\begin{eqnarray*}
u &=& \sum_{i=0}^{m} (-1)^{i}{n-i \choose i}(-1)^{l-i-1}{2l-1-(l-i-1) \choose l-i-1} \\
  &=& (-1)^{l+1}\sum_{i=0}^{l-1} {n-i \choose i}{l+i \choose l+i-(l-i-1)} \\
	&=& (-1)^{l+1}\sum_{i=0}^{l-1} {n-i \choose i}{l+i \choose 2i+1} \\
	&=& (-1)^{l+1}\sum_{i=0}^{l} {n-i \choose i}{l+i \choose 2i+1}.
\end{eqnarray*}
\noindent L'entier $v$ est la somme des produits du coefficient de degré $n-1-2i$ de $K_{n-1}(X)$ et du coefficient de $K_{2l}(X)$ de degré $2i+2$, pour tout $0 \leq i \leq {\rm min}(r,l-1)$. Par la proposition \ref{22}, on a :
\begin{eqnarray*}
v &=& \sum_{i=0}^{{\rm min}(r,l-1)} (-1)^{i}{n-1-i \choose i}(-1)^{l-i-1}{2l-(l-i-1) \choose l-i-1} \\
  &=& (-1)^{l+1}\sum_{i=0}^{l-1} {n-(i+1) \choose i}{l+(i+1) \choose l+i+1-(l-i-1)} \\
  &=& (-1)^{l+1}\sum_{i=0}^{l-1} {n-(i+1) \choose i}{l+(i+1) \choose 2(i+1)} \\
	&=& (-1)^{l+1}\sum_{j=1}^{l} {n-j \choose j-1}{l+j \choose 2j} \\
	&=& (-1)^{l+1}\sum_{i=0}^{l} {n-i \choose i-1}{l+i \choose 2i}.
\end{eqnarray*}

\noindent Comme $u=v$, on a $(-1)^{l+1}u=(-1)^{l+1}v$ et donc $\sum_{i=0}^{l} {n-i \choose i}{l+i \choose 2i+1}=\sum_{i=0}^{l} {n-i \choose i-1}{l+i \choose 2i}$.
\\
\\ii) On suppose que $2l+1>n$, c'est-à-dire $\frac{n}{2} \leq l \leq n$. On reprend la matrice $M:=M_{n}(X)[M_{2l+1}(X)]^{-1}$. On dispose de l'égalité $M=M_{2l+1-n}(X)^{-1}=\begin{pmatrix}
    -K_{2l-n-1}(X) & K_{2l-n}(X) \\
    -K_{2l-n}(X)  & K_{2l+1-n}(X)
   \end{pmatrix}$. On a donc :
	\[-K_{2l-n-1}(X)=-K_{n}(X)K_{2l-1}(X)+K_{n-1}(X)K_{2l}(X).\] 
\noindent $-K_{2l-n-1}(X)$ est un polynôme de degré $2l-n-1$. Or, comme $l \leq n$, on a $2l-n-1 \leq n-1<n+1$ et donc $K_{2l-n-1}(X)$ n'a pas de terme de degré $n+1$. On peut donc reprendre ce qui a été fait en i), en modifiant les valeurs de $m$ et de $r$. Pour obtenir la formule souhaitée, il ne reste plus qu'à rajouter des termes nuls à chacune des sommes.

\end{proof}

\subsection{Preuve de la formule générale}

La formule du théorème \ref{11} est sympathique et sa preuve nécessite peu de calculs. Cependant, pour que notre égalité puisse prendre sa forme optimale, on doit lever la condition portant sur $l$. Malheureusement, on ne peut pas accomplir cette tâche en se reposant exclusivement sur le polynôme continuant. En effet, on pourrait définir une fraction rationnelle de la forme $\sum_{i=0}^{l} (-1)^{i}{n-i \choose i}X^{n-2i}$ mais celle-ci n'est a priori relié à aucune matrice. On est donc contraint d'adopter une approche plus calculatoire. Pour mener à bien celle-ci, on va prouver plusieurs petits résultats intermédiaires qui, une fois regroupés, nous donneront la formule souhaitée.

\begin{lemma}
\label{23}

Soient $l>n$ deux entiers naturels. On dispose de l'égalité :
\[\sum_{i=0}^{n} \left[-{n-i \choose i}{l+i \choose 2i+1}+{n-i \choose i-1}{l+i \choose 2i}\right]=(-1)^{n+1}{l \choose n+1}.\]

\end{lemma}

\begin{proof}

Soient $l>n$ deux entiers naturels. On a $E\left[\frac{n}{2}\right] \leq E\left[\frac{n-1}{2}\right]+1$. On pose $\tilde{m}:=E\left[\frac{n-1}{2}\right]+1$. On utilise, comme précédemment, la matrice $M:=M_{n}(X)[M_{2l+1}(X)]^{-1}$. On dispose de l'égalité :
\[M=M_{2l+1-n}(X)^{-1}=\begin{pmatrix}
    -K_{2l-n-1}(X) & K_{2l-n}(X) \\
    -K_{2l-n}(X)  & K_{2l+1-n}(X)
   \end{pmatrix}.\]
\noindent On a donc $-K_{2l-n-1}(X)=-K_{n}(X)K_{2l-1}(X)+K_{n-1}(X)K_{2l}(X)$. Le coefficient de degré $n+1$ du polynôme de droite est $(-1)^{l+1}\sum_{i=0}^{\tilde{m}} \left[-{n-i \choose i}{l+i \choose 2i+1}+{n-i \choose i-1}{l+i \choose 2i}\right]$. En rajoutant des termes nuls, celui-ci devient :
\[(-1)^{l+1}\sum_{i=0}^{n} \left[-{n-i \choose i}{l+i \choose 2i+1}+{n-i \choose i-1}{l+i \choose 2i}\right].\]

\noindent Par ailleurs, le coefficient de degré $n+1$ de $-K_{2l-n-1}(X)$ est $(-1)^{l-n}{l \choose l-n-1}=(-1)^{l+1}\left[(-1)^{n+1}{l \choose n+1}\right]$. On a donc le résultat souhaité.

\end{proof}

On va utiliser dans tout ce qui va suivre les notations détaillées ci-dessous. Soient $n$ un entier relatif et $k$ un entier naturel non nul. On pose :
\begin{itemize}[label=$\circ$]
\item $u_{k,n}:=(-1)^{n}{n+k \choose n+1}$;
\item $v_{k,n}:=\sum_{i=n+1}^{n+k} \left[-{n-i \choose i}{n+k+i \choose 2i+1}+{n-i \choose i-1}{n+k+i \choose 2i}\right]$.
\end{itemize}
\noindent Pour démontrer l'égalité du théorème \ref{12}, on doit donc montrer que, pour tout $n \geq 0$ et $k \geq 1$, $u_{k,n}=v_{k,n}$. Cela ne semble pas évident de prime abord. Cependant, ici le membre de gauche a une forme particulièrement simple. On peut donc se concentrer d'abord sur celui-ci. Grâce au triangle de Pascal, on obtient immédiatement la formule de récurrence $u_{k+1,n+1}=u_{k,n+1}-u_{k+1,n}$, à laquelle on adjoint les conditions initiales $u_{k,0}=k$ pour tout $k \geq 1$ et $u_{1,n}=(-1)^{n}$ pour tout $n \geq 0$. Si on montre que la suite $v$ vérifie les mêmes conditions, on aura prouvé l'égalité recherchée. C'est donc à cette tâche que l'on va maintenant s'atteler.
\\
\\\indent L'une des propriétés les plus utiles lorsque l'on considère des coefficients binomiaux est le triangle de Pascal. Cependant, cette relation n'est valable que pour les coefficients binomiaux classiques. On va en donner ici un analogue pour leurs équivalents négatifs.

\begin{lemma}[\og Triangle de Pascal négatif \fg]
\label{24}

Soient $n<0$ et $i \geq 1$. L'égalité suivante est vérifiée :
\[{n \choose i}={n+1 \choose i}-{n \choose i-1}.\]

\end{lemma}

\begin{proof}

Par le triangle de Pascal, on a :
\begin{eqnarray*} 
{n \choose i} &=& (-1)^{i}{i-n-1 \choose i} \\
              &=& (-1)^{i}{i-(n+1)-1 \choose i}+(-1)^{i}{i-n-2 \choose i-1} \\
							&=& {n+1 \choose i}-(-1)^{i-1}{(i-1)-n-1 \choose i-1} \\
							&=& {n+1 \choose i}-{n \choose i-1}.
\end{eqnarray*}

\end{proof}

\begin{lemma}
\label{25}

Soient $n \geq 0$ et $k \geq 1$. On a $v_{k+1,n+1}=v_{k,n+1}-v_{k+1,n}$.

\end{lemma}

\begin{proof}

Avec le triangle de Pascal, on a les égalités ci-dessous :\begin{eqnarray*}
v_{k+1,n+1} &=& \sum_{i=n+2}^{n+k+2} -{n+1-i \choose i}\left({n+k+1+i \choose 2i+1}+{n+k+1+i \choose 2i}\right) \\
            & & +{n+1-i \choose i-1}\left({n+k+1+i \choose 2i}+{n+k+1+i \choose 2i-1}\right) \\
            &=& {-k-1 \choose n+k+1}+\sum_{i=n+2}^{n+k+1} -{n+1-i \choose i}\left({n+k+1+i \choose 2i+1}+{n+k+1+i \choose 2i}\right) \\
						& & +{n+1-i \choose i-1}\left({n+k+1+i \choose 2i}+{n+k+1+i \choose 2i-1}\right) \\
						&=& {-k-1 \choose n+k+1}+\sum_{i=n+2}^{n+k+1} \left[-{n+1-i \choose i}{n+k+1+i \choose 2i+1}+{n+1-i \choose i-1}{n+k+1+i \choose 2i}\right] \\
						& & +\sum_{i=n+2}^{n+k+1} \left[-{n+1-i \choose i}{n+k+1+i \choose 2i}+{n+1-i \choose i-1}{n+k+1+i \choose 2i-1}\right]  \\
						&=& {-k-1 \choose n+k+1}+v_{k,n+1}+\underbrace{\sum_{i=n+2}^{n+k+1} -{n+1-i \choose i}{n+k+1+i \choose 2i}}_{\alpha}+\underbrace{\sum_{i=n+2}^{n+k+1} {n+1-i \choose i-1}{n+k+1+i \choose 2i-1}}_{\beta}.
\end{eqnarray*}

\noindent En posant $j:=i-1$ et en utilisant le triangle de Pascal, on obtient :
\begin{eqnarray*}
\beta &=& \sum_{j=n+1}^{n+k} {n-j \choose j}{n+k+j+2 \choose 2j+1} \\
      &=& {-1 \choose n+1}{2n+k+3 \choose 2n+3}+\sum_{j=n+2}^{n+k} {n-j \choose j}{n+k+j+1 \choose 2j+1}+\sum_{j=n+2}^{n+k} {n-j \choose j}{n+k+j+1 \choose 2j}.
\end{eqnarray*}

\noindent En utilisant le triangle de Pascal négatif, on a :
\[\alpha=-{-k \choose n+k+1}-\sum_{i=n+2}^{n+k} {n-i \choose i}{n+k+1+i \choose 2i}-\sum_{i=n+2}^{n+k} {n-i \choose i-1}{n+k+1+i \choose 2i}.\]

\noindent En utilisant une nouvelle fois le triangle de Pascal négatif, on obtient :
\[{-k-1 \choose n+k+1}={-k \choose n+k+1}-{-k-1 \choose n+k}.\]

\noindent Par conséquent, on a, en posant $\gamma:=v_{k+1,n+1}-v_{k,n+1}$ et en utilisant que ${-1 \choose n+1}=-{-1 \choose n}$ :
\begin{eqnarray*}
\gamma &=& {-1 \choose n+1}{2n+k+3 \choose 2n+3}-{-k-1 \choose n+k}+\sum_{i=n+2}^{n+k} {n-i \choose i}{n+k+i+1 \choose 2i+1}-{n-i \choose i-1}{n+k+1+i \choose 2i} \\
       &=& {-1 \choose n+1}\left({2n+k+2 \choose 2n+3}+{2n+k+2 \choose 2n+2}\right)+\sum_{i=n+2}^{n+k+1} {n-i \choose i}{n+k+i+1 \choose 2i+1}-{n-i \choose i-1}{n+k+1+i \choose 2i} \\
			 &=& \sum_{i=n+1}^{n+k+1} {n-i \choose i}{n+k+i+1 \choose 2i+1}-{n-i \choose i-1}{n+k+1+i \choose 2i} \\
			 &=& -v_{k+1,n}.
\end{eqnarray*}
\end{proof}

\begin{lemma}
\label{26}

Soit $k \geq 1$. On a $v_{k+1,0}=1+v_{k,0}-v_{k+1,-1}$.

\end{lemma}

\begin{proof}

La preuve utilise globalement les mêmes éléments que la précédente. Aussi, on va passer certains détails. Avec le triangle de Pascal, on a :
\begin{eqnarray*}
v_{k+1,0} &=& \sum_{i=1}^{k+1} \left[-{-i \choose i}{k+1+i \choose 2i+1}+{-i \choose i-1}{k+1+i \choose 2i}\right] \\
          &=& {-k-1 \choose k}+v_{k,0}+\underbrace{\sum_{i=1}^{k} -{-i \choose i}{k+i \choose 2i}}_{\alpha'}+\underbrace{\sum_{i=1}^{k} {-i \choose i-1}{k+i \choose 2i-1}}_{\beta'}.
\end{eqnarray*}
\[\beta'=\sum_{j=0}^{k-1} {-1-j \choose j}{k+j+1 \choose 2j+1}=(k+1)+\sum_{j=1}^{k-1} {-1-j \choose j}{k+j \choose 2j+1}+\sum_{j=1}^{k-1} {-1-j \choose j}{k+j \choose 2j}.\]
\[\alpha'=-{-k \choose k}-\sum_{i=1}^{k-1} {-1-i \choose i}{k+i \choose 2i}-\sum_{i=1}^{k-1} {-1-i \choose i-1}{k+i \choose 2i}.\]

\noindent Par conséquent, on obtient :
\begin{eqnarray*}
v_{k+1,0} &=& {-k-1 \choose k}+v_{k,0}+(k+1)-{-k \choose k}+\sum_{i=1}^{k-1} -{-1-i \choose i-1}{k+i \choose 2i}+{-1-i \choose i}{k+i \choose 2i+1} \\
          &=& 1+v_{k,0}-{-k-1 \choose k-1}+\sum_{i=0}^{k-1} -{-1-i \choose i-1}{k+i \choose 2i}+{-1-i \choose i}{k+i \choose 2i+1} \\
					&=& 1+v_{k,0}-v_{k+1,-1}.
\end{eqnarray*}
\end{proof}

\begin{lemma}
\label{27}

Soient $k \geq 1$ et $n \leq -1$. On a $v_{k+1,n}=v_{k,n}-v_{k+1,n-1}$.

\end{lemma}

\begin{proof}

On distingue trois cas. Si $n \leq -k-2$ alors $v_{k+1,n}=v_{k,n}=v_{k+1,n-1}=0$ car tous les termes de ces sommes sont tous multiples d'un coefficient binomial dont l'entier inférieur est négatif. Si $n=-k$ ou $n=-k-1$ alors on a, par un calcul direct, $v_{k+1,n}=v_{k,n}=v_{k+1,n-1}=0$. On suppose maintenant que $-k+1 \leq n \leq -1$. En particulier, on a $n+k \geq 1$. La preuve utilise, là encore, des éléments analogues à ce qui a déjà été fait. 
\begin{eqnarray*}
v_{k+1,n} &=& \sum_{i=0}^{n+k+1} \left[-{n-i \choose i}{n+k+1+i \choose 2i+1}+{n-i \choose i-1}{n+k+1+i \choose 2i}\right] \\
          &=& {-k-1 \choose n+k}+v_{k,n}+\underbrace{\sum_{i=0}^{n+k} -{n-i \choose i}{n+k+i \choose 2i}}_{\alpha''}+\underbrace{\sum_{i=1}^{n+k} {n-i \choose i-1}{n+k+i \choose 2i-1}}_{\beta''}.
\end{eqnarray*}

\[\beta''=\sum_{j=0}^{n+k-1} {n-j-1 \choose j}{n+k+j+1 \choose 2j+1}=\sum_{j=0}^{n+k-1} {n-1-j \choose j}{n+k+j \choose 2j+1}+\sum_{j=0}^{n+k-1} {n-1-j \choose j}{n+k+j \choose 2j}.\]
\[\alpha''=-{-k \choose n+k}-\sum_{i=0}^{n+k-1} {n-i-1 \choose i}{n+k+i \choose 2i}-\sum_{i=0}^{n+k-1} {n-i-1 \choose i-1}{n+k+i \choose 2i}.\]

\noindent Par conséquent, on a :
\begin{eqnarray*}
v_{k+1,n} &=& v_{k,n}+{-k-1 \choose n+k}-{-k \choose n+k}+\sum_{i=0}^{n+k-1} {n-i-1 \choose i}{n+k+i \choose 2i+1}-{n-i-1 \choose i-1}{n+k+i \choose 2i} \\
          &=& -{-k-1 \choose n+k-1}+v_{k,n}+\sum_{i=n}^{n+k-1} {n-i-1 \choose i}{n+k+i \choose 2i+1}-{n-i-1 \choose i-1}{n+k+i \choose 2i} \\
          &=& v_{k,n}-v_{k+1,n-1}.
\end{eqnarray*}
\end{proof}

\begin{lemma}
\label{28}

Soient $n \geq 0$ et $k \geq 1$. On a $v_{1,n}=(-1)^{n}$ et $v_{k,0}=k$.

\end{lemma}

\begin{proof}

On considère séparément chacune des deux valeurs.
\begin{itemize}[label=$\circ$]
\item On a $v_{1,n}=-{-1 \choose n+1}{2n+2 \choose 2n+3}+{-1 \choose n}{2n+2 \choose 2n+2}={-1 \choose n}=(-1)^{n}$.
\\
\item Montrons par récurrence sur $k$ que, pour tout $h \leq -1$, $v_{k,h}=0$. Pour tout $h \leq -2$, $v_{1,h}=0$ car les termes de la somme sont tous multiples d'un coefficient binomial dont l'entier inférieur est négatif. Par ailleurs, $v_{1,-1}=0$. On suppose qu'il existe un $k \geq 1$ tel que, pour tout $h \leq -1$, $v_{k,h}=0$. Soit $h \leq -1$. Si $h \leq -k-2$ alors $v_{k+1,h}=0$, car les termes de la somme sont tous multiples d'un coefficient binomial dont l'entier inférieur est négatif. Supposons que $h \geq -k-1$. Par le lemme \ref{27}, on dispose de l'égalité $v_{k+1,-k-1}=v_{k,-k-1}-v_{k+1,-k-2}$. Or, $v_{k,-k-1}=0$, par hypothèse de récurrence, et $v_{k+1,-k-2}=0$. Donc, $v_{k+1,-k-1}=0$. Par le lemme \ref{27}, $v_{k+1,-k}=v_{k,-k}-v_{k+1,-k-1}=0$. De proche en proche, on obtient $v_{k,h}=0$.
\\
\\Ainsi, par le lemme \ref{26}, $v_{k+1,0}=1+v_{k,0}$. Montrons par récurrence que, pour tout $m \geq 1$, $v_{m,0}=m$. Si $m=1$, $v_{1,0}=(-1)^{0}=1$. Supposons qu'il existe un entier $m \geq 1$ tel que $v_{m,0}=m$. On a $v_{m+1,0}=1+v_{m,0}=m+1$. Par récurrence, le résultat est démontré.

\end{itemize}

\end{proof}

\noindent On peut maintenant effectuer la preuve de notre formule.

\begin{proof}[Démonstration du théorème \ref{12}]

Soient $n$ et $l$ deux entiers naturels. Si $0 \leq l \leq n$ alors, par le théorème \ref{11}, la formule est vraie. On suppose maintenant que $l>n$ et on écrit $l=n+k$. On pose, pour $0 \leq a \leq b \leq n+k$ :
\[S_{a,b}=\sum_{i=a}^{b} \left[-{n-i \choose i}{l+i \choose 2i+1}+{n-i \choose i-1}{l+i \choose 2i}\right].\]
\noindent On a $S_{0,n+k}=S_{0,n}+S_{n+1,n+k}$.
\\
\\Par le lemme \ref{23}, on a $S_{0,n}=(-1)^{n+1}{n+k \choose n+1}$. Par les lemmes \ref{25} et \ref{28}, la suite $(S_{n+1,n+k})_{n,k}$ vérifie la même relation de récurrence, avec les mêmes conditions initiales, que la suite $\left((-1)^{n}{n+k \choose n+1}\right)$. Donc, $S_{n+1,n+k}=(-1)^{n}{n+k \choose n+1}$. Par conséquent, on a :
\[-\sum_{i=0}^{n+k} {n-i \choose i}{n+k+i \choose 2i+1}+\sum_{i=0}^{n+k} {n-i \choose i-1}{l+i \choose 2i}=(-1)^{n}{n+k \choose n+1}(1-1)=0.\]

\end{proof}

\subsection{Valeurs numériques}
\label{num}

On donne dans le tableau ci-dessous la valeur commune de $\sum_{i=0}^{l} {n-i \choose i}{l+i \choose 2i+1}$ et $\sum_{i=0}^{l} {n-i \choose i-1}{l+i \choose 2i}$ pour différentes valeurs de $n$ et $l$.

\begin{center}
\begin{tabular}{|c|c|c|c|c|c|c|c|c|c|}
\hline
  \multicolumn{1}{|c|}{\backslashbox{$l$}{\vrule width 0pt height 1.25em$n$}} & 1 & 2 & 3 & 4 & 5 & 6 & 7 & 8   \rule[-7pt]{0pt}{18pt} \\
  \hline
  1   & 1 & 1 & 1 & 1 & 1 & 1 & 1 & 1  \rule[-7pt]{0pt}{18pt} \\
	\hline
	  2   & 2 & 3 & 4 & 5 & 6 & 7 & 8 & 9  \rule[-7pt]{0pt}{18pt} \\
	\hline
  3  & 4 & 7 & 11 & 16 & 22 & 29 & 37 & 46   \rule[-7pt]{0pt}{18pt} \\
	\hline
	  4  & 6 & 13 & 24 & 40 & 62 & 91 & 128 & 174  \rule[-7pt]{0pt}{18pt} \\
	\hline
  5  & 9 & 22 & 46 & 86 & 148 & 239 & 367 & 541  \rule[-7pt]{0pt}{18pt} \\
\hline
  6  & 12 & 34 & 80 & 166 & 314 & 553 & 920 & 1461  \rule[-7pt]{0pt}{18pt} \\
\hline
  7  & 16 & 50 & 130 & 296 & 610 & 1163 & 2083 & 3544  \rule[-7pt]{0pt}{18pt} \\
\hline
  8  & 20 & 70 & 200 & 496 & 1106 & 2269 & 4352 & 7896  \rule[-7pt]{0pt}{18pt} \\
\hline
	
\end{tabular}
\end{center}

On pose $n:=7$ et $l:=4$. Pour tout $I \subset \{0,1,2,3,4\}$, on définit les sommes $u_{I}:=\sum_{i \in I} {n-i \choose i}{l+i \choose 2i+1}$ et $v_{I}:=\sum_{i \in I} {n-i \choose i-1}{l+i \choose 2i}$. On constate, en regardant dans les tableaux ci-dessous, que $u_{I}=v_{J}$ si et seulement si $u_{I}=u_{\{0,1,2,3,4\}}$ ou $u_{I}=0$.

\begin{center}
\begin{tabular}{|c|c|c|c|c|c|c|c|c|c|c|c|c|c|c|c|}
\hline
  $I$ & $\{0\}$ & $\{1\}$ & $\{2\}$ & $\{3\}$ & $\{4\}$ & $\{0,1\}$ & $\{0,2\}$ & $\{0,3\}$ & $\{0,4\}$ & $\{1,2\}$ & $\{1,3\}$ & $\{1,4\}$ & $\{2,3\}$ & $\{2,4\}$   \rule[-7pt]{0pt}{18pt} \\
  \hline
  $u_{I}$   & 4 & 60 & 60 & 4 & 0 & 64 & 64 & 8 & 4 & 120 & 64 & 60 & 64 & 60 \rule[-7pt]{0pt}{18pt} \\
	\hline
	$v_{I}$   & 0 & 10 & 75 & 42 & 1 & 10 & 75 & 42 & 1 & 85 & 52 & 11 & 117 & 76 \rule[-7pt]{0pt}{18pt} \\
\hline
	
\end{tabular}
\end{center}

\begin{center}
\begin{tabular}{|c|c|c|c|c|c|c|c|c|c|c|}
\hline
  $I$ & $\{3,4\}$ & $\{0,1,2\}$ & $\{0,1,3\}$ & $\{0,1,4\}$ & $\{0,2,3\}$ & $\{0,2,4\}$ & $\{0,3,4\}$ & $\{1,2,3\}$ & $\{1,2,4\}$ & $\{1,3,4\}$   \rule[-7pt]{0pt}{18pt} \\
  \hline
  $u_{I}$   & 4 & 124 & 68 & 64 & 68 & 64 & 8 & 124 & 120 & 64 \rule[-7pt]{0pt}{18pt} \\
	\hline
	$v_{I}$   & 43 & 85 & 52 & 11 & 117 & 76 & 43 & 127 & 86 & 53  \rule[-7pt]{0pt}{18pt} \\
\hline
	
\end{tabular}
\end{center}

\begin{center}
\begin{tabular}{|c|c|c|c|c|c|c|c|}
\hline
  $I$ & $\{2,3,4\}$ & $\{0,1,2,3\}$ & $\{0,1,2,4\}$ & $\{0,1,3,4\}$ & $\{0,2,3,4\}$ & $\{1,2,3,4\}$ & $\{0,1,2,3,4\}$ \rule[-7pt]{0pt}{18pt} \\
  \hline
  $u_{I}$   & 64 & 128 & 124 & 68 & 68 & 124 & 128   \rule[-7pt]{0pt}{18pt} \\
	\hline
	$v_{I}$   & 118 & 127 & 86 & 53 & 118 & 128 & 128     \rule[-7pt]{0pt}{18pt} \\
\hline
	
\end{tabular}
\end{center}

\newpage

\begin{center}

\textbf{A CURIOUS EQUALITY BETWEEN TWO SUMS OF PRODUCTS OF BINOMIAL COEFFICIENTS}

\end{center}

\begin{abstract}
On va montrer dans ce texte que, pour tous entiers naturels $n$ et $l$, l'égalité suivante est vérifiée : 
\[\sum_{i=0}^{l} {n-i \choose i}{l+i \choose 2i+1}=\sum_{i=0}^{l} {n-i \choose i-1}{l+i \choose 2i}.\]
\noindent On traitera d'abord le cas où $l \leq n$ pour lequel les deux sommes ne contiennent que des coefficients binomiaux classiques. Ensuite, on se placera dans le cadre général en utilisant les coefficients binomiaux généralisés.
\end{abstract}

\selectlanguage{english}
\begin{abstract}

We will show in this text that, for all non-negative integers $n$ and $l$, the following equality is verified:
\[\sum_{i=0}^{l} {n-i \choose i}{l+i \choose 2i+1}=\sum_{i=0}^{l} {n-i \choose i-1}{l+i \choose 2i}.\]
\noindent We will first address the case where $l \leq n$, for which both sums contain only classical binomial coefficients. Then, we will consider the general framework using generalized binomial coefficients.

\end{abstract}

\begin{attentionbox}
This article was originally written in French. The English text shown below is only a translation. To access the original version, go to the page \pageref{00}.
\end{attentionbox}

\section{Introduction}
\label{introbis}

Binomial coefficients appear in many branches of mathematics. As indispensable tools in combinatorics, they arise in a myriad of well-known counting results, such as Sperner's theorem or Cayley's theorem (see for instance \cite{AZbis}). That said, they also accompany objects that may at first seem more remote from counting problems. Binomial coefficients notably appear in the expression of Chebyshev polynomials, in Sergei Bernstein’s proof of the Weierstrass approximation theorem, and also in Leibniz rule for derivatives.
\\
\\ \indent In light of this wide range of applications, it is not surprising that many properties of binomial coefficients have been investigated. Among these, two main families can be distinguished: arithmetic properties and summation formulas. In the first category, one may notably mention Lucas's theorem, which expresses the remainder of the binomial coefficient ${n \choose k}$ upon division by a prime number $p$ in terms of the base-$p$ expansions of the integers $n$ and $k$ (see \cite{Lbis}, Section XXI). As for summation formulas, there is a wide range of results, from Pascal's triangle to the binomial theorem, as well as many more or less famous identities, such as Vandermonde's identity (see \cite{Kbis}, Section 1.2.6 I).
\\
\\ \indent We propose here to follow this second approach, by proving the two formulas below using a proof that aims to be as non-computational as possible. In what follows, we adopt the following standard convention: for all $n \geq 0$, $k \geq 1$ and $0 \leq l < k$, we set ${n \choose -k}={l \choose k}=0$. Note that with this convention, Pascal’s triangle works for binomial coefficients of the form ${m \choose 0}$ and ${m \choose m}$, with $m \geq 1$. Moreover, if $x$ is a real number, we denote by $E[x]$ the integer part of $x$.

\begin{theorembis}
\label{31}

For all non-negative integers $n$ and $0 \leq l \leq n$, the following identity holds:
\[\sum_{i=0}^{l} {n-i \choose i}{l+i \choose 2i+1}=\sum_{i=0}^{l} {n-i \choose i-1}{l+i \choose 2i}.\]

\end{theorembis}

In fact, the condition $l \leq n$ is unnecessary. However, to remove it, one must give a meaning to ${n-i \choose i}$ when $n-i < 0$. We therefore define, for $n < 0$ and $k \in \mathbb{Z}$,
\[{n \choose k}:=\left\{
    \begin{array}{ll}
        (-1)^{k}{k-n-1 \choose k} & \mbox{if } k \geq 0; \\
        0 & \mbox{otherwise}.
    \end{array}
\right.  \\ \]
\noindent This definition may seem surprising at first glance, but it is justified by the identities below:
\[(-1)^{k}{k-n-1 \choose k}=(-1)^{k}\frac{(k-n-1)!}{k!(-n-1)!}=(-1)^{k}\frac{(k-n-1)\ldots(-n)}{k!}=\frac{n(n-1)\ldots(n-k+1)}{k!}.\]
\noindent Moreover, with this convention, Pascal's triangle remains valid for the binomial coefficient ${0 \choose 0}$. Thus, Pascal's triangle holds for all binomial coefficients ${n \choose k}$ with $0 \leq k \leq n$
\\
\\ \indent Equipped with this definition, we can state the result below, which provides a nice formula that, to the best of our knowledge, appears to be new:

\begin{theorembis}
\label{32}

For all non-negative integers $n$ and $l$, the following identity holds:
\[\sum_{i=0}^{l} {n-i \choose i}{l+i \choose 2i+1}=\sum_{i=0}^{l} {n-i \choose i-1}{l+i \choose 2i}.\]

\end{theorembis}

This identity may appear quite surprising at first glance, on the one hand because of the absence of conditions on $n$ and $l$, and on the other hand because of its almost perfect symmetry. Indeed, by only slightly modifying the lower indices of the binomial coefficients, there was a priori little chance of obtaining an equality. Note that this initial intuition can be further illustrated by considering subsums on each side of the identity, i.e. sums of the form $\sum_{i \in I} {n-i \choose i}{l+i \choose 2i+1}$ and $\sum_{j \in J} {n-j \choose j-1}{l+j \choose 2j}$ with $I,J \subset [\![1;l]\!]$, since in general none of the subsums on the left-hand side is equal to any of the subsums on the right-hand side (apart, of course, from the trivial subsums equal to $0$ or equal to the common value of the original sums). An example of this phenomenon is given in Section \ref{numbis}.

\section{Proof of the formula}

\subsection{The continuant polynomial}

In this subsection, we present a few elements that will be useful in what follows.
\\
\\\indent We begin by defining a well-known polynomial, originally related to continued fractions (see for instance \cite{CObis} for more details). We set $K_{-1}:=0$ and $K_{0}:=1$. Let $n$ be a positive integer and $(a_{1},\ldots,a_{n}) \in \mathbb{C}^{n}$. We denote:
\[K_{n}(a_{1},\ldots,a_{n}):=
\left|
\begin{array}{cccccc}
a_1&1&&&\\[4pt]
1&a_{2}&1&&\\[4pt]
&\ddots&\ddots&\!\!\ddots&\\[4pt]
&&1&a_{n-1}&\!\!\!\!\!1\\[4pt]
&&&\!\!\!\!\!1&\!\!\!\!a_{n}
\end{array}
\right|.\] $K_{n}(a_{1},\ldots,a_{n})$ is the continuant (or negative continuant) of $a_{1},\ldots,a_{n}$. Here, we will only need to consider the case where all the $a_{i}$ are equal, and in this case we write the continuant as $K_{n}(a_{1})$ instead of $K_{n}(a_{1},\ldots,a_{1})$. In this setting, we obtain a single-variable polynomial with integer coefficients, of degree $n$. By expanding the determinant along the first column, one immediately obtains, for all $n \geq 2$, the identity:
\begin{equation}
\label{p}
\tag{$\star$}
K_{n}(X)=X~K_{n-1}(X)-K_{n-2}(X).
\end{equation}
\noindent Note that this identity remains valid when $n=1$.
\\
\\ \indent These polynomials can be used to express the coefficients of certain matrices, which we now make explicit. Let $n$ be a positive integer and $x \in \mathbb{C}$. We set $M_{n}(x):=\begin{pmatrix}
    x & -1 \\
    1  & 0 
   \end{pmatrix}^{n}$. We have the following classical identity:

\begin{proposition}
\label{41}

Let $n$ be a positive integer and $x \in \mathbb{C}$. We have $M_{n}(x)=\begin{pmatrix}
    K_{n}(x) & -K_{n-1}(x) \\
    K_{n-1}(x)  & -K_{n-2}(x)
   \end{pmatrix}$.

\end{proposition}

\begin{proof}

We proceed by induction on $n$. If $n=1$, the result is immediate. Suppose that there exists an integer $n \geq 1$ such that the formula holds at rank $n$. We have:
\begin{eqnarray*}
M_{n+1}(x) &=& M_{n}(x)\begin{pmatrix}
    x & -1 \\
    1  & 0 
   \end{pmatrix} \\
	         &=& \begin{pmatrix}
    K_{n}(x) & -K_{n-1}(x) \\
    K_{n-1}(x)  & -K_{n-2}(x)
   \end{pmatrix}\begin{pmatrix}
    x & -1 \\
    1  & 0 
   \end{pmatrix}~~({\rm induction~hypothesis}) \\
	         &=& \begin{pmatrix}
    xK_{n}(x)-K_{n-1}(x) & -K_{n}(x) \\
    xK_{n-1}(x)-K_{n-2}(x)  & -K_{n-1}(x)
   \end{pmatrix} \\
	         &=& \begin{pmatrix}
    K_{n+1}(x) & -K_{n}(x) \\
    K_{n}(x) & -K_{n-1}(x)
   \end{pmatrix}~~({\rm equality~(\star)}).
\end{eqnarray*}

\end{proof}

In the same way as the polynomials $K_{n}(x)$, the matrices $M_{n}(x)$ are also particular cases. One can, of course, consider the more general matrices below:
\[M_{n}(a_1,\ldots,a_n):=\begin{pmatrix}
    a_{n} & -1 \\
    1  & 0 
   \end{pmatrix}\ldots\begin{pmatrix}
    a_{1} & -1 \\
    1  & 0 
   \end{pmatrix}=\begin{pmatrix}
    K_{n}(a_{1},\ldots,a_{n}) & -K_{n-1}(a_{2},\ldots,a_{n}) \\
    K_{n-1}(a_{1},\ldots,a_{n})  & -K_{n-2}(a_{2},\ldots,a_{n-1})
   \end{pmatrix}.\]
\noindent These matrices have many applications. For instance, one can show that for every $A$ in $SL_{2}(\mathbb{Z})$, there exist positive integers $n,a_{1},\ldots,a_{n}$ such that $A=M_{n}(a_{1},\ldots,a_{n})$. Since this representation is not unique, one is led to study the different representations of a given matrix, as well as the number of such representations of fixed size (see \cite{M1bis,M2bis,Obis}).
\\
\\\indent We now seek an expression in the form of a sum involving $K_{n}(X)$. To obtain it, we may use the Chebyshev polynomials of the second kind. These form a family $(U_{n})$ of polynomials in $\mathbb{Z}[X]$ satisfying, for all non-negative integer $n$, $U_{n+2}(X)=2X U_{n+1}(X)-U_{n}(X)$, with $U_{0}(X)=1$ and $U_{1}(X)=2X$. Using identity ($\star$), one observes that $K_{n}(2X)$ satisfies these same conditions. Consequently, for all $n \geq 0$, we have $K_{n}(X)=U_{n}\left(\frac{X}{2}\right)$. Using the classical expression of the polynomials $U_{n}$, we obtain the expression of $K_{n}(X)$ stated in the proposition below. However, here we aim to provide a complete proof of Theorem \ref{11}. We therefore give a detailed proof of the desired identity.

\begin{proposition}
\label{42}

Let $n \geq 0$, $K_{n}(X)=\sum_{k=0}^{E[\frac{n}{2}]} (-1)^{k}\binom{n-k}{k}X^{n-2k}$.

\end{proposition}

\begin{proof}

This can be proved by induction on $n$. Indeed, the formula holds for $n=0$ and for $n=1$. Assume that there exists a positive integer $n$ such that the formula holds for $n$ and $n-1$. We suppose that $n$ is even. Using ($\star$), we obtain:
\begin{eqnarray*}
K_{n+1}(X) &=& XK_{n}(X)-K_{n-1}(X) \\
                    &=& \sum_{k=0}^{E[\frac{n}{2}]} (-1)^{k}\binom{n-k}{k}X^{n+1-2k}-\sum_{k=0}^{E[\frac{n-1}{2}]} (-1)^{k}\binom{n-1-k}{k}X^{n-1-2k} \\
										&=&  \sum_{k=0}^{\frac{n}{2}} (-1)^{k}\binom{n-k}{k}X^{n+1-2k}-\sum_{k=0}^{\frac{n}{2}-1} (-1)^{k}\binom{n-1-k}{k}X^{n-1-2k}~({\rm since}~n~{\rm is~even}) \\
										&=&  \sum_{k=0}^{\frac{n}{2}} (-1)^{k}\binom{n-k}{k}X^{n+1-2k}-\sum_{l=1}^{\frac{n}{2}} (-1)^{l-1}\binom{n-l}{l-1}X^{n+1-2l} \\ 
										&=&  \sum_{k=0}^{\frac{n}{2}} (-1)^{k}\binom{n-k}{k}X^{n+1-2k}+\sum_{l=1}^{\frac{n}{2}} (-1)^{l}\binom{n-l}{l-1}X^{n+1-2l} \\
										&=&  X^{n+1}+\sum_{k=1}^{\frac{n}{2}} (-1)^{k}\left(\binom{n-k}{k}+\binom{n-k}{k-1}\right)X^{n+1-2k} \\
										&=&  X^{n+1}+\sum_{k=1}^{\frac{n}{2}} (-1)^{k}\binom{n+1-k}{k}X^{n+1-2k}~({\rm Pascal's~triangle}) \\
										&=&  \sum_{k=0}^{\frac{n}{2}} (-1)^{k}\binom{n+1-k}{k}X^{n+1-2k}. \\
\end{eqnarray*}

\noindent The case where $n$ is odd is treated in a similar manner. This completes the proof of the formula by induction.

\end{proof}
  
\subsection{Proof of the first theorem}

We now prove the identity announced in the introduction.

\begin{proof}[Proof of theorem \ref{31}]

If $l=0$, then both sums consist of a single term equal to $0$. If $n=0$, then $l=0$ and the equality holds. We now assume that $n,l \geq 1$ and distinguish two cases.
\\
\\i) We first assume that $2l+1 \leq n$, i.e. $1 \leq l \leq \frac{n-1}{2}$. We set
$m := {\rm min}\left(E\left[\frac{n}{2}\right], E\left[\frac{2l-1}{2}\right]\right)$ and $r := {\rm min}\left(E\left[\frac{n-1}{2}\right], l\right)$. We then have $m=l-1$ and $r=l$. By Proposition \ref{41}, we obtain the two following equalities:
\begin{eqnarray*}
M &:=& M_{n}(X)[M_{2l+1}(X)]^{-1} \\
  &=& M_{n-2l-1}(X) \\
	&=& \begin{pmatrix}
    K_{n-2l-1}(X) & -K_{n-2l-2}(X) \\
    K_{n-2l-2}(X)  & -K_{n-2l-3}(X)
   \end{pmatrix}.
\end{eqnarray*}
\begin{eqnarray*}
M &=& M_{n}(X)[M_{2l+1}(X)]^{-1} \\
  &=& \begin{pmatrix}
    K_{n}(X) & -K_{n-1}(X) \\
    K_{n-1}(X)  & -K_{n-2}(X)
   \end{pmatrix}\begin{pmatrix}
    -K_{2l-1}(X) & K_{2l}(X) \\
    -K_{2l}(X)  & K_{2l+1}(X)
   \end{pmatrix} \\
	&=& \begin{pmatrix}
    -K_{n}(X)K_{2l-1}(X)+K_{n-1}(X)K_{2l}(X) & \ldots \\
    \ldots  & \ldots
   \end{pmatrix}.
\end{eqnarray*}

\noindent Consequently, we obtain $K_{n-2l-1}(X)=-K_{n}(X)K_{2l-1}(X)+K_{n-1}(X)K_{2l}(X)$.
\\
\\Since $n-2l-1 \leq n$, the polynomial $K_{n-2l-1}(X)$ has no term of degree $n+1$. Thus, the coefficient of degree $n+1$ in $K_{n}(X)K_{2l-1}(X)$, denoted by $u$, is equal to the coefficient of degree $n+1$ in $K_{n-1}(X)K_{2l}(X)$, denoted by $v$.
\\
\\The integer $u$ is the sum of the products of the coefficient of degree $n-2i$ in $K_{n}(X)$ and the coefficient of degree $2i+1$ in $K_{2l-1}(X)$, for all $0 \leq i \leq m$. By Proposition \ref{42}, we obtain:
\begin{eqnarray*}
u &=& \sum_{i=0}^{m} (-1)^{i}{n-i \choose i}(-1)^{l-i-1}{2l-1-(l-i-1) \choose l-i-1} \\
  &=& (-1)^{l+1}\sum_{i=0}^{l-1} {n-i \choose i}{l+i \choose l+i-(l-i-1)} \\
	&=& (-1)^{l+1}\sum_{i=0}^{l-1} {n-i \choose i}{l+i \choose 2i+1} \\
	&=& (-1)^{l+1}\sum_{i=0}^{l} {n-i \choose i}{l+i \choose 2i+1}.
\end{eqnarray*}
\noindent The integer $v$ is the sum of the products of the coefficient of degree $n-1-2i$ in $K_{n-1}(X)$ and the coefficient of degree $2i+2$ in $K_{2l}(X)$, for all $0 \leq i \leq {\rm min}(r,l-1)$. By Proposition \ref{42}, we obtain:
\begin{eqnarray*}
v &=& \sum_{i=0}^{{\rm min}(r,l-1)} (-1)^{i}{n-1-i \choose i}(-1)^{l-i-1}{2l-(l-i-1) \choose l-i-1} \\
  &=& (-1)^{l+1}\sum_{i=0}^{l-1} {n-(i+1) \choose i}{l+(i+1) \choose l+i+1-(l-i-1)} \\
  &=& (-1)^{l+1}\sum_{i=0}^{l-1} {n-(i+1) \choose i}{l+(i+1) \choose 2(i+1)} \\
	&=& (-1)^{l+1}\sum_{j=1}^{l} {n-j \choose j-1}{l+j \choose 2j} \\
	&=& (-1)^{l+1}\sum_{i=0}^{l} {n-i \choose i-1}{l+i \choose 2i}.
\end{eqnarray*}

\noindent Since $u=v$, we have $(-1)^{l+1}u = (-1)^{l+1}v$, and therefore $\sum_{i=0}^{l} {n-i \choose i}{l+i \choose 2i+1}=\sum_{i=0}^{l} {n-i \choose i-1}{l+i \choose 2i}$.
\\
\\ii) We now assume that $2l+1>n$, i.e. $\frac{n}{2} \leq l \leq n$. We consider again the matrix $M:=M_{n}(X)[M_{2l+1}(X)]^{-1}$. We have the identity $M=M_{2l+1-n}(X)^{-1}=\begin{pmatrix}
    -K_{2l-n-1}(X) & K_{2l-n}(X) \\
    -K_{2l-n}(X)  & K_{2l+1-n}(X)
   \end{pmatrix}$. Thus, we obtain:
	\[-K_{2l-n-1}(X)=-K_{n}(X)K_{2l-1}(X)+K_{n-1}(X)K_{2l}(X).\] 
\noindent The polynomial $-K_{2l-n-1}(X)$ has degree $2l-n-1$. However, since $l \leq n$, we have $2l-n-1 \leq n-1 < n+1$, and therefore $K_{2l-n-1}(X)$ has no term of degree $n+1$. We may thus proceed as in i), with appropriate modifications of the values of $m$ and $r$. To obtain the desired formula, it remains only to add zero terms to each of the sums.

\end{proof}

\subsection{Proof of the general formula}

The formula of Theorem \ref{31} is pleasant, and its proof requires relatively few computations. However, in order for our identity to take its optimal form, the condition on $l$ must be removed. Unfortunately, this cannot be achieved solely by relying on the continuant polynomial. Indeed, one could define a rational expression of the form $\sum_{i=0}^{l} (-1)^{i}{n-i \choose i}X^{n-2i}$, but this is not, a priori, related to any matrix. We are therefore led to adopt a more computational approach. To carry this out, we will prove several intermediate results which, when combined, will yield the desired formula.

\begin{lemma}
\label{43}

Let $l>n$ be two non-negative integers. The following equality holds:
\[\sum_{i=0}^{n} \left[-{n-i \choose i}{l+i \choose 2i+1}+{n-i \choose i-1}{l+i \choose 2i}\right]=(-1)^{n+1}{l \choose n+1}.\]

\end{lemma}

\begin{proof}

Let $l>n$ be two non-negative integers. We have $E\left[\frac{n}{2}\right] \leq E\left[\frac{n-1}{2}\right]+1$. We set $\tilde{m} := E\left[\frac{n-1}{2}\right]+1$. As before, we consider the matrix $M := M_{n}(X)[M_{2l+1}(X)]^{-1}$. We have the identity:
\[M=M_{2l+1-n}(X)^{-1}=\begin{pmatrix}
    -K_{2l-n-1}(X) & K_{2l-n}(X) \\
    -K_{2l-n}(X)  & K_{2l+1-n}(X)
   \end{pmatrix}.\]
\noindent We therefore obtain $-K_{2l-n-1}(X)=-K_{n}(X)K_{2l-1}(X)+K_{n-1}(X)K_{2l}(X)$. The coefficient of degree $n+1$ of the right-hand side polynomial is $(-1)^{l+1}\sum_{i=0}^{\tilde{m}} \left[-{n-i \choose i}{l+i \choose 2i+1}+{n-i \choose i-1}{l+i \choose 2i}\right]$. By adding zero terms, this expression becomes:
\[(-1)^{l+1}\sum_{i=0}^{n} \left[-{n-i \choose i}{l+i \choose 2i+1}+{n-i \choose i-1}{l+i \choose 2i}\right].\]

\noindent Moreover, the coefficient of degree $n+1$ of $-K_{2l-n-1}(X)$ is $(-1)^{l-n}{l \choose l-n-1}=(-1)^{l+1}\left[(-1)^{n+1}{l \choose n+1}\right]$. We thus obtain the desired result.

\end{proof}

We will use the notation detailed below throughout what follows. Let $n$ be an integer and $k$ a positive integer. We set:
\begin{itemize}[label=$\circ$]
\item $u_{k,n}:=(-1)^{n}{n+k \choose n+1}$;
\item $v_{k,n}:=\sum_{i=n+1}^{n+k} \left[-{n-i \choose i}{n+k+i \choose 2i+1}+{n-i \choose i-1}{n+k+i \choose 2i}\right]$.
\end{itemize}
To prove the identity in Theorem \ref{32}, it is therefore enough to show that, for all $n \geq 0$ and $k \geq 1$, $u_{k,n}=v_{k,n}$. This does not appear obvious at first glance. However, the left-hand side has a particularly simple form. We can therefore begin by focusing on it. Using the Pascal's triangle, we immediately obtain the recurrence relation $u_{k+1,n+1}=u_{k,n+1}-u_{k+1,n}$, together with the initial conditions $u_{k,0}=k$ for all $k \geq 1$ and $u_{1,n}=(-1)^{n}$ for all $n \geq 0$. If we show that the sequence $v$ satisfies the same conditions, then the desired equality will follow. This is therefore the task we now undertake.
\\
\\\indent One of the most useful properties when dealing with binomial coefficients is the Pascal's triangle. However, this relation is only valid for classical binomial coefficients. Here, we provide an analogue for their negative counterparts.

\begin{lemmabis}[\og Negative Pascal's triangle \fg]
\label{44}

Let $n<0$ and $i \geq 1$. The following equality holds:
\[{n \choose i}={n+1 \choose i}-{n \choose i-1}.\]

\end{lemmabis}

\begin{proof}

By the Pascal's triangle, we have:
\begin{eqnarray*} 
{n \choose i} &=& (-1)^{i}{i-n-1 \choose i} \\
              &=& (-1)^{i}{i-(n+1)-1 \choose i}+(-1)^{i}{i-n-2 \choose i-1} \\
							&=& {n+1 \choose i}-(-1)^{i-1}{(i-1)-n-1 \choose i-1} \\
							&=& {n+1 \choose i}-{n \choose i-1}.
\end{eqnarray*}

\end{proof}

\begin{lemmabis}
\label{45}

Let $n \geq 0$ and $k \geq 1$. We have $v_{k+1,n+1}=v_{k,n+1}-v_{k+1,n}$.

\end{lemmabis}

\begin{proof}

With Pascal's triangle, we have the equalities below:\begin{eqnarray*}
v_{k+1,n+1} &=& \sum_{i=n+2}^{n+k+2} -{n+1-i \choose i}\left({n+k+1+i \choose 2i+1}+{n+k+1+i \choose 2i}\right) \\
            & & +{n+1-i \choose i-1}\left({n+k+1+i \choose 2i}+{n+k+1+i \choose 2i-1}\right) \\
            &=& {-k-1 \choose n+k+1}+\sum_{i=n+2}^{n+k+1} -{n+1-i \choose i}\left({n+k+1+i \choose 2i+1}+{n+k+1+i \choose 2i}\right) \\
						& & +{n+1-i \choose i-1}\left({n+k+1+i \choose 2i}+{n+k+1+i \choose 2i-1}\right) \\
						&=& {-k-1 \choose n+k+1}+\sum_{i=n+2}^{n+k+1} \left[-{n+1-i \choose i}{n+k+1+i \choose 2i+1}+{n+1-i \choose i-1}{n+k+1+i \choose 2i}\right] \\
						& & +\sum_{i=n+2}^{n+k+1} \left[-{n+1-i \choose i}{n+k+1+i \choose 2i}+{n+1-i \choose i-1}{n+k+1+i \choose 2i-1}\right]  \\
						&=& {-k-1 \choose n+k+1}+v_{k,n+1}+\underbrace{\sum_{i=n+2}^{n+k+1} -{n+1-i \choose i}{n+k+1+i \choose 2i}}_{\alpha}+\underbrace{\sum_{i=n+2}^{n+k+1} {n+1-i \choose i-1}{n+k+1+i \choose 2i-1}}_{\beta}.
\end{eqnarray*}

\noindent Setting $j:=i-1$ and using Pascal's triangle, we obtain:
\begin{eqnarray*}
\beta &=& \sum_{j=n+1}^{n+k} {n-j \choose j}{n+k+j+2 \choose 2j+1} \\
      &=& {-1 \choose n+1}{2n+k+3 \choose 2n+3}+\sum_{j=n+2}^{n+k} {n-j \choose j}{n+k+j+1 \choose 2j+1}+\sum_{j=n+2}^{n+k} {n-j \choose j}{n+k+j+1 \choose 2j}.
\end{eqnarray*}

\noindent Using the negative Pascal's triangle, we have:
\[\alpha=-{-k \choose n+k+1}-\sum_{i=n+2}^{n+k} {n-i \choose i}{n+k+1+i \choose 2i}-\sum_{i=n+2}^{n+k} {n-i \choose i-1}{n+k+1+i \choose 2i}.\]

\noindent Using the negative Pascal's triangle once again, we obtain:
\[{-k-1 \choose n+k+1}={-k \choose n+k+1}-{-k-1 \choose n+k}.\]

\noindent Consequently, we have, setting $\gamma:=v_{k+1,n+1}-v_{k,n+1}$ and using that ${-1 \choose n+1}=-{-1 \choose n}$:
\begin{eqnarray*}
\gamma &=& {-1 \choose n+1}{2n+k+3 \choose 2n+3}-{-k-1 \choose n+k}+\sum_{i=n+2}^{n+k} {n-i \choose i}{n+k+i+1 \choose 2i+1}-{n-i \choose i-1}{n+k+1+i \choose 2i} \\
       &=& {-1 \choose n+1}\left({2n+k+2 \choose 2n+3}+{2n+k+2 \choose 2n+2}\right)+\sum_{i=n+2}^{n+k+1} {n-i \choose i}{n+k+i+1 \choose 2i+1}-{n-i \choose i-1}{n+k+1+i \choose 2i} \\
			 &=& \sum_{i=n+1}^{n+k+1} {n-i \choose i}{n+k+i+1 \choose 2i+1}-{n-i \choose i-1}{n+k+1+i \choose 2i} \\
			 &=& -v_{k+1,n}.
\end{eqnarray*}
\end{proof}

\begin{lemmabis}
\label{46}

Let $k \geq 1$. We have $v_{k+1,0}=1+v_{k,0}-v_{k+1,-1}$.

\end{lemmabis}

\begin{proof}

The proof uses essentially the same ingredients as the previous one, so we will omit some details. Using Pascal's triangle, we have:
\begin{eqnarray*}
v_{k+1,0} &=& \sum_{i=1}^{k+1} \left[-{-i \choose i}{k+1+i \choose 2i+1}+{-i \choose i-1}{k+1+i \choose 2i}\right] \\
          &=& {-k-1 \choose k}+v_{k,0}+\underbrace{\sum_{i=1}^{k} -{-i \choose i}{k+i \choose 2i}}_{\alpha'}+\underbrace{\sum_{i=1}^{k} {-i \choose i-1}{k+i \choose 2i-1}}_{\beta'}.
\end{eqnarray*}
\[\beta'=\sum_{j=0}^{k-1} {-1-j \choose j}{k+j+1 \choose 2j+1}=(k+1)+\sum_{j=1}^{k-1} {-1-j \choose j}{k+j \choose 2j+1}+\sum_{j=1}^{k-1} {-1-j \choose j}{k+j \choose 2j}.\]
\[\alpha'=-{-k \choose k}-\sum_{i=1}^{k-1} {-1-i \choose i}{k+i \choose 2i}-\sum_{i=1}^{k-1} {-1-i \choose i-1}{k+i \choose 2i}.\]

\noindent Hence, we obtain:
\begin{eqnarray*}
v_{k+1,0} &=& {-k-1 \choose k}+v_{k,0}+(k+1)-{-k \choose k}+\sum_{i=1}^{k-1} -{-1-i \choose i-1}{k+i \choose 2i}+{-1-i \choose i}{k+i \choose 2i+1} \\
          &=& 1+v_{k,0}-{-k-1 \choose k-1}+\sum_{i=0}^{k-1} -{-1-i \choose i-1}{k+i \choose 2i}+{-1-i \choose i}{k+i \choose 2i+1} \\
					&=& 1+v_{k,0}-v_{k+1,-1}.
\end{eqnarray*}
\end{proof}

\begin{lemmabis}
\label{47}

Let $k \geq 1$ and $n \leq -1$. We have $v_{k+1,n}=v_{k,n}-v_{k+1,n-1}$.

\end{lemmabis}

\begin{proof}

We distinguish three cases. If $n \leq -k-2$, then $v_{k+1,n}=v_{k,n}=v_{k+1,n-1}=0$, since all the terms in these sums are multiples of a binomial coefficient whose lower index is negative. If $n = -k$ or $n=-k-1$, then a direct computation gives $v_{k+1,n}=v_{k,n}=v_{k+1,n-1}=0$. We now assume that $-k+1 \leq n \leq -1$. In particular, we have $n+k \geq 1$. The proof again relies on arguments analogous to those already used.
\begin{eqnarray*}
v_{k+1,n} &=& \sum_{i=0}^{n+k+1} \left[-{n-i \choose i}{n+k+1+i \choose 2i+1}+{n-i \choose i-1}{n+k+1+i \choose 2i}\right] \\
          &=& {-k-1 \choose n+k}+v_{k,n}+\underbrace{\sum_{i=0}^{n+k} -{n-i \choose i}{n+k+i \choose 2i}}_{\alpha''}+\underbrace{\sum_{i=1}^{n+k} {n-i \choose i-1}{n+k+i \choose 2i-1}}_{\beta''}.
\end{eqnarray*}

\[\beta''=\sum_{j=0}^{n+k-1} {n-j-1 \choose j}{n+k+j+1 \choose 2j+1}=\sum_{j=0}^{n+k-1} {n-1-j \choose j}{n+k+j \choose 2j+1}+\sum_{j=0}^{n+k-1} {n-1-j \choose j}{n+k+j \choose 2j}.\]
\[\alpha''=-{-k \choose n+k}-\sum_{i=0}^{n+k-1} {n-i-1 \choose i}{n+k+i \choose 2i}-\sum_{i=0}^{n+k-1} {n-i-1 \choose i-1}{n+k+i \choose 2i}.\]

\noindent Thus, we have:
\begin{eqnarray*}
v_{k+1,n} &=& v_{k,n}+{-k-1 \choose n+k}-{-k \choose n+k}+\sum_{i=0}^{n+k-1} {n-i-1 \choose i}{n+k+i \choose 2i+1}-{n-i-1 \choose i-1}{n+k+i \choose 2i} \\
          &=& -{-k-1 \choose n+k-1}+v_{k,n}+\sum_{i=n}^{n+k-1} {n-i-1 \choose i}{n+k+i \choose 2i+1}-{n-i-1 \choose i-1}{n+k+i \choose 2i} \\
          &=& v_{k,n}-v_{k+1,n-1}.
\end{eqnarray*}
\end{proof}

\begin{lemma}
\label{48}

Let $n \geq 0$ and $k \geq 1$. We have $v_{1,n}=(-1)^{n}$ and $v_{k,0}=k$.

\end{lemma}

\begin{proof}

We consider each of the two values separately.
\begin{itemize}[label=$\circ$]
\item We have $v_{1,n}=-{-1 \choose n+1}{2n+2 \choose 2n+3}+{-1 \choose n}{2n+2 \choose 2n+2}={-1 \choose n}=(-1)^{n}$.
\\
\item We prove by induction on $k$ that, for all $h \leq -1$, $v_{k,h}=0$. For all $h \leq -2$, $v_{1,h}=0$ since the terms in the sum are all multiples of a binomial coefficient whose lower index is negative. Moreover, $v_{1,-1}=0$. Assume that there exists $k \geq 1$ such that, for all $h \leq -1$, $v_{k,h}=0$. Let $h \leq -1$. If $h \leq -k-2$, then $v_{k+1,h}=0$, since the terms in the sum are all multiples of a binomial coefficient whose lower index is negative. Suppose now that $h \geq -k-1$. By Lemma \ref{47}, we have the identity $v_{k+1,-k-1}=v_{k,-k-1}-v_{k+1,-k-2}$. Now, $v_{k,-k-1} 0$ by the induction hypothesis, and $v_{k+1,-k-2}=0$. Hence $v_{k+1,-k-1}=0$. Again by Lemma \ref{47}, $v_{k+1,-k}=v_{k,-k}-v_{k+1,-k-1}=0$. Proceeding step by step, we obtain $v_{k,h}=0$.
\\
\\Thus, by Lemma \ref{46}, $v_{k+1,0}=1+v_{k,0}$. We prove by induction that, for all $m \geq 1$, $v_{m,0}=m$. If $m=1$, then $v_{1,0}=(-1)^0=1$. Assume there exists an integer $m \geq 1$ such that $v_{m,0}=m$. Then we have $v_{m+1,0}=1+v_{m,0}=m+1$. By induction, the result follows.

\end{itemize}

\end{proof}

\noindent We can now carry out the proof of our formula.

\begin{proof}[Proof of theorem \ref{32}]

Let $n$ and $l$ be two non-negative integers. If $0 \leq l \leq n$, then by Theorem \ref{31}, the formula holds. We now assume that $l > n$ and write $l = n + k$. We set, for $0 \leq a \leq b \leq n + k$:
\[S_{a,b}=\sum_{i=a}^{b} \left[-{n-i \choose i}{l+i \choose 2i+1}+{n-i \choose i-1}{l+i \choose 2i}\right].\]
\noindent We have $S_{0,n+k}=S_{0,n}+S_{n+1,n+k}$.
\\
\\By Lemma \ref{43}, we have $S_{0,n} = (-1)^{n+1}{n+k \choose n+1}$. By Lemmas \ref{45} and \ref{48}, the sequence $(S_{n+1,n+k}){n,k}$ satisfies the same recurrence relation, with the same initial conditions, as the sequence $\left((-1)^{n}{n+k \choose n+1}\right)$. Hence, $S{n+1,n+k}=(-1)^{n}{n+k \choose n+1}$. Therefore, we obtain:
\[-\sum_{i=0}^{n+k} {n-i \choose i}{n+k+i \choose 2i+1}+\sum_{i=0}^{n+k} {n-i \choose i-1}{l+i \choose 2i}=(-1)^{n}{n+k \choose n+1}(1-1)=0.\]

\end{proof}

\subsection{Numerical values}
\label{numbis}

We give in the table below the common value of $\sum_{i=0}^{l} {n-i \choose i}{l+i \choose 2i+1}$ and $\sum_{i=0}^{l} {n-i \choose i-1}{l+i \choose 2i}$ for different values of $n$ and $l$.

\begin{center}
\begin{tabular}{|c|c|c|c|c|c|c|c|c|c|}
\hline
  \multicolumn{1}{|c|}{\backslashbox{$l$}{\vrule width 0pt height 1.25em$n$}} & 1 & 2 & 3 & 4 & 5 & 6 & 7 & 8   \rule[-7pt]{0pt}{18pt} \\
  \hline
  1   & 1 & 1 & 1 & 1 & 1 & 1 & 1 & 1  \rule[-7pt]{0pt}{18pt} \\
	\hline
	  2   & 2 & 3 & 4 & 5 & 6 & 7 & 8 & 9  \rule[-7pt]{0pt}{18pt} \\
	\hline
  3  & 4 & 7 & 11 & 16 & 22 & 29 & 37 & 46   \rule[-7pt]{0pt}{18pt} \\
	\hline
	  4  & 6 & 13 & 24 & 40 & 62 & 91 & 128 & 174  \rule[-7pt]{0pt}{18pt} \\
	\hline
  5  & 9 & 22 & 46 & 86 & 148 & 239 & 367 & 541  \rule[-7pt]{0pt}{18pt} \\
\hline
  6  & 12 & 34 & 80 & 166 & 314 & 553 & 920 & 1461  \rule[-7pt]{0pt}{18pt} \\
\hline
  7  & 16 & 50 & 130 & 296 & 610 & 1163 & 2083 & 3544  \rule[-7pt]{0pt}{18pt} \\
\hline
  8  & 20 & 70 & 200 & 496 & 1106 & 2269 & 4352 & 7896  \rule[-7pt]{0pt}{18pt} \\
\hline
	
\end{tabular}
\end{center}

We set $n:=7$ and $l:=4$. For all $I \subset \{0,1,2,3,4\}$, we define the sums $u_{I}:=\sum_{i \in I} {n-i \choose i}{l+i \choose 2i+1}$ and $v_{I}:=\sum_{i \in I} {n-i \choose i-1}{l+i \choose 2i}$. We observe, by looking at the tables below, that $u_{I}=v_{J}$ if and only if $u_{I}=u_{\{0,1,2,3,4}\}$ or $u_{I}=0$.

\begin{center}
\begin{tabular}{|c|c|c|c|c|c|c|c|c|c|c|c|c|c|c|c|}
\hline
  $I$ & $\{0\}$ & $\{1\}$ & $\{2\}$ & $\{3\}$ & $\{4\}$ & $\{0,1\}$ & $\{0,2\}$ & $\{0,3\}$ & $\{0,4\}$ & $\{1,2\}$ & $\{1,3\}$ & $\{1,4\}$ & $\{2,3\}$ & $\{2,4\}$   \rule[-7pt]{0pt}{18pt} \\
  \hline
  $u_{I}$   & 4 & 60 & 60 & 4 & 0 & 64 & 64 & 8 & 4 & 120 & 64 & 60 & 64 & 60 \rule[-7pt]{0pt}{18pt} \\
	\hline
	$v_{I}$   & 0 & 10 & 75 & 42 & 1 & 10 & 75 & 42 & 1 & 85 & 52 & 11 & 117 & 76 \rule[-7pt]{0pt}{18pt} \\
\hline
	
\end{tabular}
\end{center}

\begin{center}
\begin{tabular}{|c|c|c|c|c|c|c|c|c|c|c|}
\hline
  $I$ & $\{3,4\}$ & $\{0,1,2\}$ & $\{0,1,3\}$ & $\{0,1,4\}$ & $\{0,2,3\}$ & $\{0,2,4\}$ & $\{0,3,4\}$ & $\{1,2,3\}$ & $\{1,2,4\}$ & $\{1,3,4\}$   \rule[-7pt]{0pt}{18pt} \\
  \hline
  $u_{I}$   & 4 & 124 & 68 & 64 & 68 & 64 & 8 & 124 & 120 & 64 \rule[-7pt]{0pt}{18pt} \\
	\hline
	$v_{I}$   & 43 & 85 & 52 & 11 & 117 & 76 & 43 & 127 & 86 & 53  \rule[-7pt]{0pt}{18pt} \\
\hline
	
\end{tabular}
\end{center}

\begin{center}
\begin{tabular}{|c|c|c|c|c|c|c|c|}
\hline
  $I$ & $\{2,3,4\}$ & $\{0,1,2,3\}$ & $\{0,1,2,4\}$ & $\{0,1,3,4\}$ & $\{0,2,3,4\}$ & $\{1,2,3,4\}$ & $\{0,1,2,3,4\}$ \rule[-7pt]{0pt}{18pt} \\
  \hline
  $u_{I}$   & 64 & 128 & 124 & 68 & 68 & 124 & 128   \rule[-7pt]{0pt}{18pt} \\
	\hline
	$v_{I}$   & 118 & 127 & 86 & 53 & 118 & 128 & 128     \rule[-7pt]{0pt}{18pt} \\
\hline
	
\end{tabular}
\end{center}


\begin{thebibliography}{99}

\bibitem{AZ}
M. Aigner, G. Ziegler,
{\it Proofs from THE BOOK}, Second edition,
Springer-Verlag Berlin Heidelberg GmbH, 2001.

\bibitem{CO}
C.~Conley, V.~Ovsienko, 
{\it Rotundus: triangulations, Chebyshev polynomials, and Pfaffians}, 
The Mathematical Intelligencer, Vol. 40 no. 3, (2018), pp 45-50.

\bibitem{K} 
D. E. Knuth,
{\it The Art of Computer Programming: Volume 1: Fundamental Algorithms (3rd ed.)}, 
Addison-Wesley Professional, 1997.

\bibitem{L}
É. Lucas,
{\it Théorie des Fonctions Numériques Simplement Périodiques},
American Journal of Mathematics, Vol. 1 no. 3, (1878), pp 197-240.

\bibitem{M1}
F. Mabilat,
\textit{Quelques éléments de combinatoire des matrices de $SL(2,\mathbb{Z})$}, 
Bulletin des Sciences Mathématiques, Vol. 167, Article 102958, (2021), https://doi.org/10.1016/j.bulsci.2021.102958.

\bibitem{M2}
F. Mabilat,
\textit{Éléments de comptage sur les générateurs du groupe modulaire et les $\lambda$-quiddités}, (2025), arXiv:2502.01328.

\bibitem{O} 
V. Ovsienko, 
{\it Partitions of unity in $SL(2,\mathbb{Z})$,  negative continued fractions,  and dissections of polygons,} 
Research in the Mathematical Sciences, Vol. 5 no. 2, (2018), Article 21, 25 pp.


\end{thebibliography}

\begin{thebibliography}{99}

\bibitem{AZbis}
M. Aigner, G. Ziegler,
{\it Proofs from THE BOOK}, Second edition,
Springer-Verlag Berlin Heidelberg GmbH, 2001.

\bibitem{CObis}
C.~Conley, V.~Ovsienko, 
{\it Rotundus: triangulations, Chebyshev polynomials, and Pfaffians}, 
The Mathematical Intelligencer, Vol. 40 no. 3, (2018), pp 45-50.

\bibitem{Kbis} 
D. E. Knuth,
{\it The Art of Computer Programming: Volume 1: Fundamental Algorithms (3rd ed.)}, 
Addison-Wesley Professional, 1997.

\bibitem{Lbis}
É. Lucas,
{\it Théorie des Fonctions Numériques Simplement Périodiques},
American Journal of Mathematics, Vol. 1 no. 3, (1878), pp 197-240.

\bibitem{M1bis}
F. Mabilat,
\textit{Quelques éléments de combinatoire des matrices de $SL(2,\mathbb{Z})$}, 
Bulletin des Sciences Mathématiques, Vol. 167, Article 102958, (2021), https://doi.org/10.1016/j.bulsci.2021.102958.

\bibitem{M2bis}
F. Mabilat,
\textit{Éléments de comptage sur les générateurs du groupe modulaire et les $\lambda$-quiddités}, (2025), arXiv:2502.01328.

\bibitem{Obis} 
V. Ovsienko, 
{\it Partitions of unity in $SL(2,\mathbb{Z})$,  negative continued fractions,  and dissections of polygons,} 
Research in the Mathematical Sciences, Vol. 5 no. 2, (2018), Article 21, 25 pp.




\end{thebibliography}
\end{document}